\documentclass[12pt,leqno,a4paper]{amsart}
\usepackage{amssymb, amsmath, enumerate,}
\usepackage{enumerate,color}

\newtheorem{theorem}{Theorem}[section]
\newtheorem{lemma}[theorem]{Lemma}

\newtheorem*{thmA}{Theorem A}
\newtheorem*{thmB}{Theorem B}
\newtheorem*{thmC}{Theorem C}
\newtheorem*{thmD}{Theorem D}

\theoremstyle{remark}
\newtheorem{remark}[theorem]{Remark}
\newtheorem{example}[theorem]{Example}
\numberwithin{equation}{section}

\begin{document}
\title[$p$-nilpotent maximal subgroups in finite groups]{$p$-nilpotent maximal subgroups in finite groups}
    \author[Beltr\'an and Shao]{Antonio Beltr\'an\\
     Departamento de Matem\'aticas\\
      Universitat Jaume I \\
     12071 Castell\'on\\
      Spain\\
     \\Changguo Shao \\
College of Science\\ Nanjing University of Posts and Telecommunications\\
     Nanjing 210023 Yadong\\
      China\\
     }

 \thanks{Antonio Beltr\'an: abeltran@uji.es ORCID ID: https://orcid.org/0000-0001-6570-201X \newline
 \indent Changguo Shao: shaoguozi@163.com ORCID ID: https://orcid.org/0000-0002-3865-0573}



\begin{abstract}
Let $p$ be a prime number and suppose that  every  maximal subgroup of a finite group  is either $p$-nilpotent or has prime index.
Such group  need not be $p$-solvable, and  we study its structure by proving that only one nonabelian simple group of order divisible by $p$,
 which belongs to the family  ${\rm PSL}_n(q)$, can be involved in it.
 For $p=2$, we specify more, and in fact, such simple group must be isomorphic to ${\rm PSL}_2({r^a})$ for certain values of the prime $r$ and the parameter $a$.
\end{abstract}

\maketitle

\section{Introduction}

Throughout this paper, all groups are supposed to be finite, and we follow standard notation (e.g.
 \cite{Robinson}).
\medskip

It is widely acknowledged that the structure of a group can be significantly affected by certain properties of its maximal subgroups.
 For instance, the solvability and structure of minimal non-nilpotent groups \cite[Theorem 9.1.9]{Robinson},
  called Schmidt groups, or the classification of minimal non-solvable groups
   are renowned illustrations of how the structure of a group is influenced or determined by the properties of its maximal subgroups.
   In recent years, the class of Schmidt groups has been extended in several directions, for instance,
   by requiring  all non-nilpotent maximal subgroups of a group to be normal \cite{LG, LPZ}, or to have prime index \cite{YJK, SLT}, respectively.
    The solvability of these groups and further properties are achieved through these generalizations.

\medskip
 Let $p$ be a prime number. Recall that a group is said to be $p$-nilpotent provided that it has a normal Hall
$p'$-subgroup.   N. It\^o  established that  minimal non-$p$-nilpotent groups are indeed Schmidt groups \cite[IV.5.4]{Hup}, and consequently,
are solvable groups.
In particular, groups all whose maximal subgroups are $p$-nilpotent are $p$-solvable, and it turns out that they have $p$-length $1$.
 Our objective in this paper is  to continue extending the class of minimal non-$p$-nilpotent groups
  by exploring the structure of those groups whose maximal subgroups are either $p$-nilpotent or have prime index.
  Of course, such groups need not be $p$-solvable; ${\rm Alt}(5)$ and ${\rm Sym}(5)$ are examples of it for $p=2$; and ${\rm PSL}_3(2)$ and ${\rm PGL}_3(2)$ work for $p=3$.
   Our first result determines the structure of such groups when $p=2$ in the non-solvable case.

\begin{thmA}
Let $G$ be a non-solvable finite group and suppose that every  maximal subgroup of   $G$ is either $2$-nilpotent or has prime index in $G$.
Let  $S(G)$ denote the solvable radical of $G$.  Then $S(G)={\bf O}_{2',2}(G)$ and one of the three  possibilities holds:

\begin{itemize}

\item[(i)]  $G/{\bf O}_{2',2}(G)\cong {\rm Alt}(5)\cong {\rm PSL}_2(5)$  or ${\rm PSL}_2(7)$ or ${\rm PSL}_2(11)$;

\item[(ii)] $G/{\bf O}_{2',2}(G)\cong {\rm Sym}(5)$ or ${\rm PGL}_2(7)$;

\item[(iii)]${\bf O}^2(G)<G$ and  ${\bf O}^{2}(G)/{\bf O}_{2',2}(G)\cong  S\times \ldots \times  S,$
  where $S$ is isomorphic to  ${\rm PSL}_2(p^{2^a})$, for some odd prime $p$ and $a\geq 1$, or for some odd prime $p\equiv \pm 1$ $(mod~ 8)$  and $a=0$.
  \end{itemize}
\end{thmA}

For proving Theorem A, we make use of several results based on the Classification of Finite Simple Groups, more precisely,
 Guralnick's classification of simple groups having  some subgroup of prime power index
 and a more recent and extensive result  concerning the structure of groups whose non-solvable maximal subgroups have prime power index  \cite[Theorem 1]{DM}.
  Likewise, we require information on the subgroup structure and the maximal subgroups of some specific simple groups and some families of simple groups,
  for which we will appeal to different sources \cite{BHR, Con, Dickson, Hup, Liebeck, Liebeck2}.
   We want to mention that the authors determined in \cite{SB} the structure of the non-solvable groups all whose maximal subgroups are $2$-nilpotent or normal.
   This class of groups is, however,  more restrictive, if we take into account that there are no simple groups satisfying such conditions.

\medskip
When $p=2$ the solvable case is, of course, much easier to deal with. We establish the $2$-length of the solvable groups verifying the  hypotheses in Theorem A. In fact, we obtain something else.

 \begin{thmB}
 Let $G$ be a solvable finite group and suppose that every  maximal subgroup of   $G$ is either $2$-nilpotent or has prime index. Then ${\bf O}_{2',2,2',2}(G)=G$.
 \end{thmB}

The case $p$ odd is challenging  due to  the fact that  there are more simple groups  in  Guralnick's classification that must be taken into account.
 Indeed, depending on the value of the odd prime $p$, there exist different simple groups (all isomorphic to certain ${\rm PSL}_n(q)$)
  whose maximal subgroups are either $p$-nilpotent or have prime index. We show several examples of this in Section 4.
   Nevertheless, we  obtain the following restrictions in the structure of the non-$p$-solvable  groups satisfying our conditions.

\begin{thmC}
Let $G$ be a non-$p$-solvable finite group and suppose that every  maximal subgroup of   $G$ is either $p$-nilpotent or has prime index in $G$.
Let  $S_p(G)$ denote the $p$-solvable radical of $G$.  Then

\begin{itemize}
  \item[(i)]
  $G/S_5(G)\cong {\rm PSL}_2(11)$ and $p=5$;
  \item[(ii)] or $G/S_p(G)\cong {\rm PSL}_n(2^d)$, where $n\geq3$ is a prime and $d\geq 1$ is odd such that $(n,2^d-1)=1$,
   and $p$ is a primitive prime divisor of $2^{d(n-1)}-1$ and $\frac{2^{dn}-1}{2^d-1}$ is a prime.
      \end{itemize}
\end{thmC}

In the $p$-solvable  case, we obtain the following property concerning the $p$-structure.

\begin{thmD}
 Let  $p$ an odd prime and $G$ a $p$-solvable finite group.
 If every  maximal subgroup of $G$ is $p$-nilpotent or has prime index in $G$, then  ${\bf O}_{p',p,p',p,p'}(G)=G$. In particular, $l_p(G)\leq 2$.
 \end{thmD}

\section{Preliminaries}

\medskip

The first result that we state, mentioned in the Introduction,  is  Guralnick's classification of nonabelian simple groups having a subgroup of prime power index.
As we said, this is based on the Classification  of Finite Simple Groups and  is crucial for our development.

\begin{lemma} {\normalfont  \cite[Theorem 1]{Gura}} \label{g}
\ Let $G$ be a nonabelian simple group with $H < G$ and
$|G : H| = p^a$, $p$ prime. One of the following holds.

\begin{itemize}

\item[$(a)$] $G={\rm Alt}(n)$, and $H\cong {\rm Alt}(n-1)$, with $n= p^a$.

\item[$(b)$] $G = {\rm PSL}_n(q)$ and $H$ is the stabilizer of a line or hyperplane. Then
$|G: H|= (q^n-1)/(q - 1) = p^a$. (Note that $n$ must be prime).

\item[$(c)$] $G={\rm PSL}_2(11)$ and $H\cong {\rm Alt}(5)$.

\item[$(d)$] $G=M_{23}$ and $H\cong M_{22}$ or $G=M_{11}$ and $H\cong M_{10}$.

\item[$(e)$] $G= {\rm PSU}_4(2)\cong {\rm PSp}_4(3)$ and $H$ is the parabolic subgroup of
index $27$.
\end{itemize}
\end{lemma}

The next result, also quoted in the Introduction, extends Guralnick's classification and is essential to prove Theorem A.

\begin{theorem}{\label{2}} \cite[Theorem 1]{DM}  Let $G$ be a non-solvable group in which every non-solvable maximal subgroup has
prime power index. Then:
\begin{itemize}

\item[(i)] the non-abelian composition factors of the group $G$ are pairwise isomorphic and are exhausted
by groups from the following list:
 \begin{itemize}
\item[(1)] ${\rm PSL}_2(2^p)$, where $p$ is a prime,
\item[(2)] ${\rm PSL}_2(3^p)$, where $p$ is a prime,
\item[(3)]  ${\rm PSL}_2(p^{2^a})$, where $p$ is an odd prime and $a \geq 0$,

\item[(4)] ${\rm Sz}(2^p)$, where $p$ is an odd prime,

\item[(5)] ${\rm PSL}_3(3)$;
\end{itemize}
\item[(ii)] for any simple group $S$ from the list in statement $({\rm i})$, there exists a group $G$ such that any
of its non-solvable maximal subgroups has primary index and $Soc(G)\cong S$.
\end{itemize}
\end{theorem}

In the next lemma,  we  detail the structure of the  Sylow normalizers of ${\rm PSL}_2(q)$ with $q$ a prime power  for the  convenience of readers, as it will be  frequently used.

\begin{lemma}\label{lemma2}  Let $G={\rm PSL}_2(q)$, where $q$ is a power of prime $p$ and $d = (2, q + 1)$. Let $r$ be a prime divisor of $|G|$ and $R\in {\rm Syl}_r(G)$.

\begin{itemize}

\item[(1)] If $r=p$, then ${\bf N}_G(R)=R\rtimes C_{\frac{q-1}{d}}$;

\item[(2)] If $2\neq r\mid \frac{q+1}{d}$, then ${\bf N}_G(R)=C_{\frac{q+1}{d}}\rtimes C_2$;

\item[(3)] If $2\neq r\mid \frac{q-1}{d}$, then ${\bf N}_G(R)=C_{\frac{q-1}{d}}\rtimes C_2$;

\item[(4)] Assume $p\neq r=2$.

\begin{itemize}

\item[(4.1)] If $q\equiv\pm1$ $(mod~ 8)$, then ${\bf N}_G(R)=R$;

\item[(4.2)] If $q\equiv\pm3$ $(mod~ 8)$, then ${\bf N}_G(R)={\rm Alt}(4)$.

\end{itemize}
\end{itemize}
\end{lemma}

\begin{proof} This follows from \cite[Theorem 2.8.27]{Hup}.
\end{proof}

\section{The case $p=2$}

The purpose of this section is to prove Theorems A and B.

\begin{proof}[Proof of Theorem A]
Let  $\overline{G}:=G/S(G)> 1$. Assume first that every maximal subgroup of $\overline{G}$ is $2$-nilpotent.
Since $\overline{G}$ cannot be $2$-nilpotent because it is non-solvable, then $\overline{G}$ is a minimal non-$2$-nilpotent group.
 But  then, by a theorem of It\^o \cite[Theorem IV.5.4]{Hup}, the minimal non-$2$-nilpotent groups are minimal non-nilpotent groups, so $\overline{G}$ would be solvable, a contradiction.
  Therefore, we can assume that $\overline{G}$ has a maximal subgroup $\overline{M_1}$ that is not $2$-nilpotent, so $M_1$ has prime index in $G$ by hypothesis.
  On the other hand, since $\overline{G}$ is non-solvable, not every maximal subgroup of $\overline{G}$ can have prime index in $\overline{G}$,
   otherwise $\overline{G}$ would be supersolvable by a well-known result of Huppert.
   Then, if $\overline{M_2}$ is one of such subgroups,  again by hypothesis, $M_2$ must be $2$-nilpotent.
   We conclude that a maximal subgroup that
has prime index in $G$ and a $2$-nilpotent maximal subgroup both exist and contain $S(G)$.
In particular, $S(G)$ is $2$-nilpotent, and hence $S(G)$ is  contained in ${\bf O}_{2',2}(G)$.
The converse containment follows by Feit-Thompson Theorem. Thus, ${\bf O}_{2',2}(G)=S(G)$ (possibly trivial).
 Henceforth, without loss of generality, we will assume $S(G)=1$.

\medskip
  We claim that $G$ possesses a unique minimal normal subgroup, say $N$.
  Assume on the contrary that there is another minimal normal subgroup $K$ of $G$. Since $S(G)=1$, then $K$ is non-solvable.
  Now, we  take $M$ to be a $2$-nilpotent maximal subgroup of $G$, which we know that exists.
   The non-solvability of $N$ implies that $G=NM$. As $K\cap N=1$ then
 $$K\cong KN/N\leq G/N\cong M/(N\cap M),$$
 from which we deduce that $K$ is 2-nilpotent, a contradiction. Therefore, the claim is proved.

\medskip

   We can write $N=S_1\times \cdots \times S_n$, where $S_i$ are isomorphic  nonabelian simple groups and $n\geq 1$.
   Let $M$ be a non-solvable maximal subgroup of $G$, which obviously is not $2$-nilpotent. Then the hypothesis implies that $M$ has prime index in $G$.
    In particular, we have that $G$ satisfies the hypotheses of Theorem \ref{2}.
    Observe that we are also including the case in which all maximal subgroups of $G$ are solvable, and this happens exactly when $n=1$ and $S_1=N=G$ is a minimal simple group. We conclude that $S_i$ (for every $i$) is isomorphic to one of the groups,  say $S$,  listed in the thesis of Theorem \ref{2}.  On the other hand, the fact that ${\bf C}_G(N)=1$ leads to $N\leq G\leq {\rm Aut}(N)$. Next, we distinguish two
    possibilities: $N=G$ and $N<G$.
    We will see that the first  possibility gives rise  to case (i), while  the second leads to cases (ii) and (iii) of the theorem.

\medskip
 {\bf Case 1}. Assume $N=G$, or equivalently, $G$ is nonabelian simple.

 \medskip
 Since $G$ has a prime index subgroup,  then we can combine Guralnick's classification and Theorem \ref{2},  and obtain exactly three possibilities: $G\cong {\rm Alt}(5)$, $G\cong {\rm PSL}_2(11)$ or $G\cong {\rm PSL}_2(q)$ with $q$ a prime power such that $(q^2-1)/(q-1)= q+1 $ is a prime power too according to Lemma \ref{g}(2). We distinguish two cases. Assume first that $q$ is odd. Then certainly $q$ is a Mersenne prime (of course, $q>3$) and moreover,  if $q\neq 7$, we have $q\equiv 1$ (mod $8$), and in that case it is well-known that ${\rm PSL}_2(q$) has a maximal subgroup isomorphic to Sym$(4)$, which is not $2$-nilpotent nor has prime index, that is, ${\rm PSL}_2(q)$ does not satisfy the hypothesis. Then the remaining  possibility for $G$ is ${\rm PSL}_2(7)\cong {\rm PSL}_3(2)$. We will discuss this case in the following. Assume now that $q$ is a power of $2$. By  Theorem \ref{2}(1), we know that $q=2^r$ with $r$ prime. However, $2^r +1$ is divisible by $3$ whenever $r$ is odd. Therefore,  $r= 2$ or $2^r+1$ is a power of $3$, and by Catalan conjecture we have $r=3$. It follows that the only possibilities for $G$ are  ${\rm PSL}_2(5)\cong {\rm Alt}(5)$ or ${\rm PSL}_3(2)\cong {\rm PSL}_2(7)$ (which already appeared above).
Thus, in this case we  conclude $G\cong {\rm Alt}(5)$ or ${\rm PSL}_2(7)$ or  ${\rm PSL}_2(11)$ and all these groups meet the hypotheses of the theorem. Hence this gives rise to (i).

 \medskip
{\bf  Case 2}. $N<G$. We distinguish two subcases depending on the $2$-nilpotency of the maximal subgroups that do not contain $N$.

 \medskip
 2.1) Suppose that there exists a maximal subgroup $M$ such that $N\not\leq M$ (which necessarily exists) such that $M$ is not $2$-nilpotent.

  \medskip
 By hypothesis $|G:M|=r$ is prime. The uniqueness of $N$ implies core$_G(M)=1$, so $G\leq {\rm Sym}(r)$.
  Also the maximality of $M$ gives $NM=G$, and hence $|N:N\cap M|=r$.
  Then $r$ is the largest prime dividing $|G|$ (and $|N|$), and $|G|_r=r$. As a consequence $n=1$, that is, $N$ is simple.
   Recall the $N$ is one of the simple groups listed in Theorem \ref{2}, so by applying again Guralnick's classification  we obtain as in case 1 that  $N\cong {\rm Alt}(5)$ or $N\cong {\rm PSL}_2(7)$ or $N\cong {\rm PSL}_2(11)$. Also, we know that $N< G\leq {\rm Aut}(N)$.  Next we discuss  these cases.

 If $N \cong {\rm Alt}(5)$, then $G\cong {\rm Aut}({\rm Alt}(5))\cong {\rm Sym}(5)$, and notice that ${\rm Sym}(5)$ satisfies the hypotheses of the theorem.
 The same happens with $N\cong {\rm PSL}_2(7)$ because  Out$(N)\cong C_2$ and $G\cong {\rm Aut}({\rm PSL}_2(7))\cong {\rm PGL}_2(7)$ also satisfies the hypotheses.
  Finally, if $N\cong {\rm PSL}_2(11)$, then Out$(N)\cong C_2$ and $G\cong {\rm PGL}_2(11)$.
  However, $ {\rm PGL}_2(11)$ has maximal subgroups isomorphic to ${\rm Sym}(4)$, which  do not have prime index, contradicting the hypothesis.
  Thus, this gives case (ii).

 \medskip
2.2)  Suppose that every maximal subgroup of $G$ that does not contain $N$ is $2$-nilpotent.

\medskip
   First we claim that the normalizer in $G$ of every Sylow subgroup of $N$ is $2$-nilpotent and  in particular, the normalizer in $S$ of every Sylow subgroup of $S$ is $2$-nilpotent.
   Indeed, let $P$ be a Sylow $p$-subgroup of $N$ for an arbitrary prime $p$ dividing $|N|$.
   By the  Frattini argument,   $G={\bf N}_G(P)N$, and then there exists a maximal subgroup of $G$ that contains ${\bf N}_G(P)$ and does  not contain $N$.
   Then, our assumption  implies that such maximal subgroup is $2$-nilpotent.
   In particular, ${\bf N}_G(P)$ and ${\bf N}_{N}(P)=\prod{\bf N}_S(P_0)$  where $P_0=P\cap S\in {\rm Syl}_p(S)$, are $2$-nilpotent as well, so the assertion is proved.

\medskip
   In the following, we do a  case-by-case analysis for the distinct possibilities for $S$ appearing in Theorem \ref{2}.

\medskip
 We start our analysis by assuming that $S\cong {\rm  PSL}_2(2^q)$ with $q$ prime, and  consider the prime $2$.
  Lemma \ref{lemma2}(1) asserts that if $P_0\in {\rm Syl}_2(S)$, then ${\bf N}_S(P_0)\cong P_0\rtimes C_{2^q-1}$, which is not $2$-nilpotent,
   contradicting the above property, so this case can be discarded.
   A similar argument works to reject the Suzuki group, ${\rm Sz}(q)$, with $q=2^{2n+1}$,
    because the normalizer of a Sylow $2$-subgroup $P$ is isomorphic to $P\rtimes C_{q-1}$ \cite[Chap XI, Theorem 3.10]{HB}.

 \medskip

   For the group $S\cong {\rm PSL}_3(3)$, observe that all Sylow normalizers are $2$-nilpotent, so the above argument cannot be applied  to discard it.
   However,  we check in \cite{Con} that $S$ has a unique conjugacy class of maximal subgroups of order $24$, isomorphic to ${\rm Sym}(4)$.
   Then the direct product of $n$ such subgroups, each one in a distinct $S_i$,  constitute a conjugacy class of subgroups of $N$.
   Take $H_0\cong {\rm Sym}(4)$ a maximal subgroup  of $S$  and put $H=H_0 \times\ldots\times H_0\leq N$.
   Then the Frattini argument applies to get $G={\bf N}_G(H)N$.
   Indeed, if $g\in G$, since $G$ (transitively) permutes the factors $S_i$, then  $H^g=\prod H_0^g$,
   where each factor belongs to a distinct $S_i$, is maximal in such $S_i$ and isomorphic to ${\rm Sym}(4)$.
    Accordingly, $H^g$ is the direct product of $n$  maximal subgroups  isomorphic to ${\rm Sym}(4)$, and lies in the above-mentioned conjugacy class of subgroups of $N$.
    It follows that $H^g=H^n$ for some $n\in N$, and hence $G={\bf N}_G(H)N$, as wanted.
    Now, if we take a maximal subgroup $M$ of $G$ containing ${\bf N}_G(H)<G$, we clearly have $N\not\leq M$ and thus, by  our assumption, $M$ is $2$-nilpotent.
    Hence ${\bf N}_G(H)$  should be $2$-nilpotent too, contradicting the fact that $H$ is not.
    Then this case can be eliminated.

\medskip
Suppose now that $S\cong {\rm PSL}_2(3^r)$ with $r$ prime. If $r$ is odd, then  $3^r\equiv 3$ (mod $8$).
 Also, if  $P_0\in {\rm Syl}_2(S)$,   then  ${\bf N}_S(P_0)\cong {\rm Alt}(4)$  by Lemma \ref{lemma2}(4.1).
 As ${\rm Alt}(4)$  is not $2$-nilpotent, these cases are excluded too.
  Therefore, the only remaining case within this family is $r=2$, that is, ${\rm PSL}_2(3^2)$, which belongs to  case (3) of Theorem \ref{2}. this will be discussed below.

  \medskip
     Finally, assume that $S\cong {\rm PSL}_2(q^{2^a})$ with $q$ an odd prime and $a\geq 0$.
      The Sylow normalizers of this group are  $2$-nilpotent  for all primes (see again Lemma \ref{lemma2})   except for the prime $2$ and $q^{2^a}\equiv \pm 3$ (mod $8$);
      in that case the normalizers of the Sylow $2$-subgroups of $S$ are  isomorphic to ${\rm Alt}(4)$, which is not $2$-nilpotent.
      But the above  congruence only occurs  when $a=0$ and $q\equiv \pm 3$ (mod  $8$), because if $a\geq 1$, then $q^{2^a}\equiv 1$ (mod $8$) for every odd prime $q$.
        Thus,  in this family of simple groups, only the case $a=0$ and  $q\equiv \pm 3$  (mod  $8$) can be discarded.
      We want to stress that, according to  the list of subgroups of ${\rm PSL}_2(q^{2^a})$ with $a\geq 0$, given in Dickson's book \cite{Dickson}
       (see also \cite[Theorem 2.1]{King} for a specific list of conjugacy classes of subgroups),
        we know that ${\rm PSL}_2(q^{2^a})$ for each $a\geq 1$  does not have any single conjugacy class of subgroups (for every  isomorphic type) of non-$2$-nilpotent subgroups.
          Thus,  a similar argument to that of ${\rm PSL}_3(3)$ cannot be applied here so as to rule out more groups within this family of simple groups.

\medskip
       This conclude our discussion on simple groups, so we have proved that $S$ can only be isomorphic to one of the simple groups appearing in case (iii) of the theorem.

\medskip
 The rest of the proof of this case consists in proving that $G/N$ is a $2$-group. We know that
 $$N< G \leq {\rm Aut}(N)= {\rm Aut}(S)\wr {\rm Sym}(n),$$
 where   ${\rm Aut}(S)\wr {\rm Sym}(n)$ denotes the wreath product.
  Write $A={\rm Aut}(S)$ and  let $A^*$ be the base group of $A\wr {\rm Sym}(n)$.
  In order to prove that $G/N$ is a $2$-group we show that $(G\cap A^*)/N$ and  $G/(G\cap A^*)$  are $2$-groups.
  For the first group , we know that
  $$(G\cap A^*)/N\leq {\rm Out}(S)\times\ldots\times {\rm Out}(S).$$
   As we have proved above that $S\cong {\rm PSL}_2(q^{2^a})$ with $q$ an odd prime, then it is known that Out$(S)$ has order $(2, q^{2^a}-1) 2^a= 2^{a+1}$ \cite{Con},
    so our first assertion follows. It remains to show that $G/(G\cap A^*$) is a $2$-group as well.
      Notice that $$G/(G\cap A^*)\cong A^*G/A^*=(A^*G\cap {\rm Sym}(n))A^*/A^*\cong A^*G\cap {\rm Sym}(n).$$
Suppose that there is a prime $p\neq 2$ such that $p$ divides $|A^*G\cap {\rm Sym}(n)|$ and seek a contradiction.
Let $P_0\in {\rm Syl}_2(S)$, so $P=P_0\times \ldots \times P_0\in {\rm Syl}_2(N)$ and the Frattini argument gives $G={\bf N}_G(P)N$.
 Now, since $|G/(G\cap A^*)|$ is divisible by $p$ and $N\leq G\cap A^*$, then there exists a $p$-element $x\in G\setminus G\cap A^*$ such that $x\in {\bf N}_G(P)$.
 Also, we can write $x= as$, with $a\in A^*$ and $1\neq s\in {\rm Sym}(n)$.
 It is straightforward that $s$ normalizes $P$, and  consequently $a\in {\bf N}_{A^*}(P)={\bf N}_A(P_0)\times \ldots \times {\bf N}_A(P_0)$.
 Since $s$ permutes non-trivially the direct factors of $P$, in particular $x$ does not centralize $P$.
 But by the first paragraph of the proof of this case, we know that ${\bf N}_G(P)$ is $2$-nilpotent, and this forces $x$ to centralize $P$, which is a contradiction.
 This shows that $G/(G\cap A^*)$ is a $2$-group, as wanted.
 Moreover, it follows that ${\bf O}^2(G)=N$, and since we are assuming $N<G$, the proof is finished.
\end{proof}

Before proving Theorem B, we recall that if $p$ is prime and $G$ is a finite  $p$-solvable group,
 then there exists a normal series in $G$ whose factors are $p$-groups or $p'$-groups.
 The minimal number of factors in any such series that are $p$-groups is called the $p$-length of $G$ and is denoted by $l_p(G)$.
 It is not hard to prove that $l_p(G)$ is equal to the number of $p$-groups in the upper $p'p$-series  of $G$ (and in the lower $p'p$-series too) \cite[Theorem 9.1.4]{Robinson}.
   Thus, Theorem B establishes the $2$-length of the groups under study.

\begin{proof}[Proof of Theorem B]
Let $K={\bf O}_{2',2,2'}(G)$ and  assume ${\bf O}_{2',2}(G)<K<G$, otherwise the result is proved.
 If $M$ is any maximal subgroup of $G$ containing $K$, it is clear that $M$ cannot be $2$-nilpotent because $K$ neither is, so by hypothesis $|G:M|$ must be prime.
 Hence, by a well-known theorem of Huppert, it follows that $G/K$ is supersolvable, so in particular, $G/K$ is $2$-nilpotent.
 It follows then that $G/K$ is a $2$-group, so we conclude that ${\bf O}_{2',2,2',2}(G)=G$, as required.
\end{proof}

\begin{example} We show that  the conclusion  ${\bf O}_{2',2,2',2}(G)=G$ in Theorem B is sharp. Let us consider the group $G=C_3 \rtimes {\rm Sym}(4)$, that is,
 the semidirect product of $C_3$ and ${\rm Sym}(4)$ acting via ${\rm Sym}(4)/{\rm Alt}(4)\cong C_2$. In fact, $G={\sf SmallGroup}(72,43)$ taken from the {\sf SmallGroups} Library of GAP \cite{GAP}.
  This group has the following upper $2' 2$-series:
$$1< {\bf O}_{2'}(G)= C_3 < {\bf O}_{2',2}(G)=C_2\times C_6 <{\bf O}_{2',2,2'}(G)=C_3 \times {\rm Alt}(4) <{\bf O}_{2',2,2',2}(G)=G.$$
Moreover, the maximal subgroups of $G$ are  isomorphic to either ${\rm Sym(}4)$,   $C_3 \times {\rm Alt}(4)$, $C_3\rtimes D_4$ or $C_3\rtimes {\rm Sym}(3)$.
All of them have prime index in $G$ except $C_3\rtimes {\rm Sym}(3)$, which is $2$-nilpotent.
\end{example}

\begin{remark}
By employing a well-known theorem due to Hall-Higman \cite[Satz VI.6.6(a)]{Hup},
 an immediate consequence of  Theorem B is that the nilpotency class of the Sylow $2$-subgroups of any group that satisfies the assumptions of the theorem  is at most $2$.
\end{remark}

\section{The case $p$ odd}

We start with an application of the Glauberman-Thompson $p$-nilpotency criterion for an odd prime $p$  \cite[Satz IV.6.2]{Hup}, which is useful for our purposes.
We stress, however, that  the following lemma is not true when $p=2$: the Mathieu group $M_{10}$ with a  minimal normal subgroup isomorphic to ${\rm Alt}(6)$ is an example of this,
 since every maximal subgroup of $M_{10}$ distinct from ${\rm Alt}(6)$ is $2$-nilpotent. This can be checked in \cite{Con}.

\begin{lemma}\label{maximal} Let $G$ be a finite group, $p$ an odd prime and $N$ a  minimal normal subgroup of $G$ that is not $p$-solvable.
 Then there exists a maximal subgroup in $G$ that does not contain $N$ and is not $p$-nilpotent.
\end{lemma}

\begin{proof}
Notice that $N$ must be a direct product of isomorphic nonabelian simple groups. Let $P_0$ be a Sylow $p$-subgroup of $N$, which is non-trivial for $N$ being non-$p$-solvable.
 Let $J(P_0)\neq 1$ be the Thompson subgroup of $P_0$ (see  \cite[IV.6.1.]{Hup} for a definition). Clearly ${\bf Z}(J(P_0))$ is not normal in $G$ since cannot be normal in $N$,
  and accordingly, ${\bf N}_G({\bf Z}(J(P_0))) $ is a  proper
 subgroup of $G$ lying in some maximal subgroup, $M$, of $G$.
  Now, if $M$ is $p$-nilpotent,  so is ${\bf N}_N({\bf Z}(J(P_0)))$, and by  Glauberman-Thompson's criterion, we would have that $N$ is $p$-nilpotent too, a contradiction.
   Hence, $M$ is not $p$-nilpotent. On the other hand, the Frattini argument gives ${\bf N}_G(P_0)N=G$.
    Moreover, since ${\bf Z}(J(P_0))$ is characteristic in $P_0$, we have
 $$ {\bf N}_G(P_0)\leq {\bf N}_G({\bf Z}(J(P_0)))\leq M,$$
 so we conclude that $M$ does not contain $N$ either.
 \end{proof}

A feature of finite group theory is that many theorems on finite groups can only be
proved by reducing them to a check that some property holds for all or certain finite simple groups.
This is what we do through the following theorem,  by analyzing first almost simple groups that meet the hypotheses of Theorem C.

\begin{theorem} \label{almostsimple}
Let $p$ be an odd prime, $S$ be a finite nonabelian simple group of order divisible by $p$ and $S \leq G \leq {\rm Aut}(S)$.
 Assume that every maximal subgroup of $G$ is either $p$-nilpotent or has prime index. Then  either
  \begin{itemize}
  \item[(1)]
  $G\cong {\rm PSL}_2(11)$ and $p=5$;

   \item[(2)] or $G\cong {\rm PSL}_n(2^d)$, where $n\geq3$ is a prime and $d\geq 1$ is odd such that $(n,2^d-1)=1$,
    and $p$ is a primitive prime divisor of $2^{d(n-1)}-1$ and $\frac{2^{dn}-1}{2^d-1}$ is a prime.
  \end{itemize}
\end{theorem}

\begin{proof}
By Lemma \ref{maximal}, we know that $G$ has a maximal subgroup $M$ that is not $p$-nilpotent and does not contain $S$.
 Then, by hypothesis $M$ must have prime index in $G$, say $|G:M|=q$. Also,
 the maximality of $M$ gives $MS=G$, and then $|S:S\cap M|=|G:M|=q$.
  Furthermore, since core$_G(M)=1$, we have $G\leq {\rm Sym}(q)$ and note that $q$ is the largest prime dividing $|G|$.
  Now, taking into account Guralnick's classification, Lemma \ref{g}, we have the following possibilities:

\begin{itemize}

\item[(a)] $S \cong  M_{11}$, $S\cap M\cong  M_{10}$, with $q = 11$; or $S \cong M_{23}$, $S\cap M \cong  M_{22}$ and $q = 23$;

\item[(b)] $S\cong {\rm PSL}_2(11)$, $S \cap M \cong {\rm Alt}(5)$ with $q=11$;

\item[(c)] $S \cong {\rm Alt}(q)$, $S \cap M \cong {\rm Alt}(q-1)$;

\item[(d)] $S \cong {\rm PSL}_n(s)$,  $q = \frac{s^n-1}{s-1} > 3$, where $n$ is some prime.

\end{itemize}

Next, we do a case-by-case analysis to prove that only case (b) with $G=S$ and case (d) with $n>2$ can occur.

\medskip

  For both sporadic groups of case (a)  it is known that Out$(S)=1$, so $G=S$.
  Then, by using the Atlas \cite{Con}, we check  that $M_{11}$ has a maximal subgroup of index $12$ isomorphic to ${\rm PSL}_2(11)$,
   which obviously is not $r$-nilpotent for any prime in $\pi(M_{11})=\{2,3,5,11\}$.
   Hence, in this case $G$ does not satisfy the hypotheses of the theorem.
   Similarly, for $M_{23}$ we find,  among others, maximal subgroups isomorphic to $M_{11}$, ${\rm Alt}(8)$ or $C_{23} \rtimes C_{11}$.
   Then we can ensure that for any prime $r\in \pi(M_{23})=\{2,3,5,7,11,23\}$, the group $M_{23}$ possesses  a maximal subgroup that is not $r$-nilpotent.
   Thus, in this case the hypothesis does not hold either.

  \medskip
   In case (b), Out$(S)\cong C_2$ and then $G={\rm PSL}_2(11)$ or $G={\rm PGL}_2(11)$. In the former case, the maximal subgroups of $G$,  apart from ${\rm Alt}(5)$ of index $11$, are: either isomorphic to $C_{11}\rtimes C_5$, which is not $11$-nilpotent; or isomorphic to $D_{12}$, which is not $3$-nilpotent. Both types of maximal subgroups have composite index in $G$, so this case can be excluded  for $p= 3$ and $11$. We remark that this argument cannot be applied  to the prime $p=5$. In fact, every maximal subgroup of $G$ is $5$-nilpotent or has prime index, so this case gives rise to  the group appearing in case (i) of the theorem.
   Assume now that $G={\rm PGL}_2(11)$. Then the maximal subgroups of  $G$ are isomorphic to $C_{11}\rtimes C_{10}$, $D_{20}$, $D_{24}$ or ${\rm Sym}(4)$.
   But none of them has prime index, and for each prime $r\in \pi(G)$, at least one of them is not $r$-nilpotent. Therefore, this case does not meet the hypotheses, and is excluded as well.

\medskip

  In case (c), we have Out$(S)\cong C_2$ and $G={\rm Alt}(q)$ or $G={\rm Sym}(q)$.
   First of all we  show that  ${\rm Alt}(q)$ possesses a maximal subgroup that neither has prime index nor is $p$-nilpotent,  contradicting the hypothesis.
     Suppose first that $q\geq 7$ and let us consider $T=({\rm Sym}(q-2)\times {\rm Sym}(2))\cap {\rm Alt}(q)$, the twisted symmetric subgroup on $q-2$ letters.
    This subgroup of ${\rm Alt}(q)$ is maximal (see \cite{Liebeck}) and is isomorphic to ${\rm Sym}(q-2)$.
    Hence, $T$ is not $r$-nilpotent for any prime $r<q$. In particular, $T$ is not $p$-nilpotent whenever $p<q$.
    Moreover, $T$, which has order $(q-2)!$, has index $q(q-1)/2$ in ${\rm Alt}(q)$, which obviously is not a prime number.
    Therefore,  the case $p<q$ is finished. It remains to study the case $p=q$, that is, we must show that ${\rm Alt}(q)$ possesses a non-$q$-nilpotent maximal subgroup whose index is not prime.
    It is enough to consider a maximal subgroup $M$  containing ${\bf N}_{{\rm Alt}(q)}(Q)$, where $Q\in {\rm Syl}_q({\rm Alt}(q))$.
     Certainly $M$ has not prime  index, because a maximal subgroup  of prime index must have index $q$.
     Next we see that it cannot be $q$-nilpotent either. Otherwise  ${\bf N}_{{\rm Alt}(q)}(Q)$ would be $q$-nilpotent too.
     However, being $q$ the largest prime dividing $|{\rm Alt}(q)|$, it is known that ${\bf N}_{{\rm Alt}(q)}(Q)\cong C_q\rtimes C_{\frac{q-1}{2}}$, which is not $q$-nilpotent, so we have a contradiction.
   Finally, suppose $q=5$.
   Then ${\rm Alt}(5)$ has maximal subgroups  of index $10$, isomorphic to ${\rm Sym}(3)$, which are not $3$-nilpotent,
    and besides has maximal subgroups of index $6$, isomorphic to $D_{10}$, which are not $5$-nilpotent.
    Therefore, the case ${\rm Alt}(q)$  for every prime $q\geq 5$ is discarded.

\medskip
  We study now the case ${\rm Sym}(q)$.
   Assume first that $q\geq 7$.
   It is enough to consider the maximal subgroup ${\rm Sym}(q-2)\times {\rm Sym}(2)$ (see again \cite{Liebeck}), which is not $r$-nilpotent for every prime $r<q$,
    so in particular, is not $p$-nilpotent if $p<q$, and of course, has not prime index.
    For  $p=q$, we  proceed as above with ${\rm Alt}(q)$, by taking a maximal subgroup of ${\rm Sym}(q)$ containing  ${\bf N}_{{\rm Sym}(q)}(Q)$, with
$Q\in {\rm Syl}_q({\rm Sym}(q))$.
 Then we have ${\bf N}_{{\rm Sym}(q)}(Q)\cong C_q\rtimes C_{q-1}$, which is not $q$-nilpotent,
 so similarly as above we get a contradiction.
  Finally, if $q=5$,  then ${\rm Sym}(5)$ has maximal subgroups  isomorphic to ${\rm Sym}(3)\times {\rm Sym}(2)$, which are not $3$-nilpotent,
   and maximal subgroups  isomorphic to $C_5\rtimes C_4$, which are not $5$-nilpotent.
     Moreover, none of them has prime index. We conclude that ${\rm Sym}(q)$ does not satisfy the hypothesis for every prime $q\geq 5$, as required.

  \medskip
 Finally, we analyze case (d). Recall that the order of the Singer cycle of ${\rm PSL}_n(s)$ is equal to $$\frac{s^n-1}{(n,s-1)(s-1)}=\frac{q}{(n,s-1)},$$ and it is an integer.
 The fact that $q$ is prime yields two possibilities: either $(n,s-1)=q$ or $1$. If $(n,s-1)=q$, then  $q|(s-1)$.
 But $q=(s^n -1)/(s-1)=s^{n-1}+ \ldots +s+1$,  and this certainly leads to a contradiction.  Therefore,  $(n,s-1)=1$. Write $s=r^d$, where $r$ is  a prime and $d$ is a positive integer.
  We distinguish two cases: $n=2$ and  $n\geq 3$.

\medskip
 Assume first that $n=2$, so $r^d+1=q$. By elementary number theory, it follows that $r= 2$, so $S={\rm PSL}_2(2^d)$.
Let $Q$ be a Sylow $q$-subgroup of $S$ (hence, of $G$ too since $G\leq {\rm Sym}(q)$).
 Then the Frattini argument gives $G={\bf N}_S(Q)S$.
 Since $Q$ is not normal in $G$, there exists a maximal  subgroup $L$ of $G$ containing ${\bf N}_G(Q)$, which obviously does not contain $S$.
 Now, by Lemma \ref{lemma2}(2) and (3), we observe that ${\bf N}_S(Q)$ is not $p$-nilpotent (for any odd prime $p$) and hence, $L$ cannot be $p$-nilpotent either.
 Then, according to the hypothesis  $|G:L|$ should be strictly less  than $q$ (prime number), contradicting the fact that $G\leq {\rm Sym}(q)$. This contradiction excludes case $n=2$.

\medskip
  For the remaining cases, that is, when $n\geq 3$, we observe that, since $n\neq 2,6$,
  then  $q$ is a primitive prime divisor of $s^n-1$.
  Therefore, if $Q$ is a Sylow $q$-subgroup of $S$ (and of $G$), by \cite[Theorem 2.7.3]{Hup}, we have that ${\bf N}_S(Q)=Q\rtimes C_n$ is a Frobenius group (not $q$-nilpotent).
   As a consequence, if $p=q$, then  the hypothesis implies that every maximal subgroup of $G$ containing ${\bf N}_S(Q)$ should have prime index (less than $q$),
    but  this is impossible.
    Consequently, in the following, we will assume  $p\neq q$.

 Since $(n,s-1)=1$, we recall that Out$({\rm PSL}_n(r^d))$ is isomorphic to $[C_d]C_2 $
  and is generated  by  the field automorphism $\phi$ of order $d$ and the graph automorphism $\gamma$ of order $2$.
  It is easy to see that $p\neq r$.
  Assume  first that $G/S$ is  not contained in $\langle \phi \rangle$.
  Then there is some conjugate of $\gamma$ in $G/S$. Without loss of generality, we may assume that $\gamma\in G/S$.
  Assume first that $p=n$.
   By \cite[Tables]{BHR} and \cite[Proposition 4.1.17]{Liebeck2}, we have that $G$ has a maximal subgroup ${\rm SL}_{n-1}(s).T$, where $T=(G/S)\cap {\rm Out}({\rm PSL}_n(s))$.
    Note that $|G:{\rm SL}_{n-1}(s).T|$ is not a prime and  $p$ divides $|{\rm SL}_{n-1}(s)|$, so ${\rm SL}_{n-1}(s)$ should be $p$-nilpotent, a contradiction.
     Hence, we can assume $p\neq n$. Next we prove that $p$ is a primitive prime divisor of $s^n-1$. On the contrary, assume that
      $p\mid (s^i-1)$ for some $1\leq i\leq n-1$.
     By \cite[Tables]{BHR} and \cite[Proposition 4.1.17]{Liebeck2}, ${\rm SL}_{n-1}(s).T$ is maximal in $G$, where $T=(G/S)\cap {\rm Out}({\rm PSL}_n(s))$.
     It is easy to see that the index of ${\rm SL}_{n-1}(s).T$ in $G$ is not a prime.
      Hence ${\rm SL}_{n-1}(s).T$ should be $p$-nilpotent, in particular, ${\rm SL}_{n-1}(s)$ should be too. As $p$ divides $|{\rm SL}_{n-1}(s)|$, this certainly gives a contradiction.
      Henceforth, we will assume that $p$ is a primitive prime divisor of $s^n-1$.

\
By \cite[Tables]{BHR} and \cite[Proposition 4.1.17]{Liebeck2}, we know that $(\frac{s^n-1}{s-1}:n).T$ is maximal in $G$, where $T=(G/S)\cap {\rm Out}({\rm PSL}_n(s))$.
It is easy to see that $|G:(\frac{s^n-1}{s-1}:n).T|$ is not a prime, and consequently $(\frac{s^n-1}{s-1}:n).T$ is $p$-nilpotent. Thus $\frac{s^n-1}{s-1}:n$ should be $p$-nilpotent too.
 However, by \cite[Theorem 2.7.3]{Hup},   we have that $(\frac{s^n-1}{s-1})_p:n$ is a Frobenius group, where $(\frac{s^n-1}{s-1})_p$ is the Sylow $p$-subgroup of $\frac{s^n-1}{s-1}$,
  so we get a contradiction.
   By \cite[Tables]{BHR} and \cite[Proposition 4.1.17]{Liebeck2} again, we also can get that ${\rm SL}_{n-1}(s)$ is $p$-nilpotent and this leads to a contradiction.

 Therefore we can assume that   $G/S$ is  contained in $\langle \phi \rangle$.
     Then by \cite[Tables]{BHR} and \cite[Proposition 4.1.17]{Liebeck2},
      we know that $S$ has a maximal  parabolic subgroup $N_1$ of type $P_2$ which is invariant under $\phi$,  such that ${\rm PSL}_2(s)\times {\rm PSL}_{n-2}(s)\leq N_1$.
      Then, $N_1\langle \phi \rangle\cap G$ is maximal in $G$.  It is easy to see that
$$\frac{(s^n-1)(s^{n-1}-1)}{(n,s-1)(s^2-1)}\mid |G:N_1\langle \phi\rangle\cap G |.$$
 In particular, this index  is not a prime.
 Thus, by hypothesis, $N_1$ should be $p$-nilpotent, and hence,   ${\rm PSL}_{n-2}(s)$ should be $p$-nilpotent too.
  Therefore, if  $p$ divides $|{\rm PSL}_{n-2}(s)|$, we clearly obtain a contradiction.
  We can assume then that $p$ does not  divide $|{\rm PSL}_{n-2}(s)|$, in particular, $p$ does not divide $s^{i}-1$  for any $1\leq i\leq n-2$ either.
  As $p$ does not divide $(s^n-1)/(s-1)=q$, it follows that $p$ is a primitive prime divisor of $s^{n-1}-1$.
   Assume that $s$ is odd. By \cite[Tables]{BHR} and \cite[Proposition 4.1.17]{Liebeck2}, we have that $G$ has a maximal subgroup of the form ${\rm PSO}_n(s).T$,
   where $T=(G/S)\cap \langle\phi \rangle$.
   It is easy to see that the index of ${\rm PSO}_n(s).T$ in $G$ is not a prime and $p$ does not divide $|{\rm PSO}_n(s)|$. We can get a contradiction as above.
   Hence $s$ is power of 2, say $s=2^t$. Assume that $t$ is even. By \cite[Tables]{BHR} and \cite[Proposition 4.1.17]{Liebeck2} again,
    we have that $G$ has a maximal subgroup of form ${\rm PSU}_n(s_0).T$, where $s=s_0^2$ and $T=(G/S)\cap {\rm Out}({\rm PSL}_n(s))$.
    Similarly, we can also get a contradiction.
    Therefore,  we deduce that $t$ is odd.
    Finally, we prove that $(G/S)\cap \langle \phi \rangle=1$, that is $G=S$ is a nonabelian simple group.
    If $(G/S)\cap \langle \phi \rangle >1$, then we know that ${\rm  SL}_{n-1}(s).T$ is maximal in $G$ again by \cite[Tables]{BHR} and \cite[Proposition 4.1.17]{Liebeck2},
     where $T=(G/S)\cap \langle \phi \rangle$.
    It is easy to see that $|G:{\rm  SL}_{n-1}(s).T|$ is not a prime.
     It follows that ${\rm  SL}_{n-1}(s).T$ is $p$-nilpotent, so is ${\rm SL}_{n-1}(s)$.
     Since $p$ divides $|{\rm  SL}_{n-1}(s)|$, we get a contradiction.
     Hence $G=S\cong {\rm PSL}_n(2^d)$, where $d\geq1$ is an odd number.
      Moreover, as we said at the beginning of proof of case (d), we know that $\frac{s^n-1}{s-1}=q$ is a prime.
      Therefore, we obtain all  the conditions of case $(2)$ of the statement.
 \end{proof}

We are ready to prove Theorem C.

\begin{proof}[Proof of Theorem C]
Let  $\overline{G}:=G/S_p(G)> 1$. Assume first that every maximal subgroup of $\overline{G}$ is $p$-nilpotent.
 Of course, $\overline{G}$ is not.  Then, by a theorem of It\^o,  quoted in the Introduction, it is known that minimal non-$p$-nilpotent groups are minimal non-nilpotent groups,
  so in particular are solvable. But this leads to the $p$-solvability of $G$,  contradicting our hypothesis.
  Consequently, we can assume that $\overline{G}$ has a maximal subgroup $\overline{M_1}$ that is not $p$-nilpotent, so $M_1$ has prime index in $G$.
   On the other hand, since $\overline{G}$ is non-$p$-solvable, not every maximal subgroup of $\overline{G}$ can have prime index,  otherwise $\overline{G}$ would be supersolvable,
    which provides a contradiction too.
    Thus, if $\overline{M_2}$ is one of such subgroups, then  again by hypothesis, $M_2$ must be $p$-nilpotent.
    We  conclude that at least both, a maximal subgroup of prime index and a $p$-nilpotent maximal subgroup  exist in $G$, and both subgroups contain $S_p(G)$.
    Henceforth, and without loss of generality, we will assume   $S_p(G)=1$.

\medskip
  We claim that $G$ has exactly one minimal normal subgroup, say $N$.
   Assume on the contrary that there is another minimal normal subgroup $K$ of $G$. Since $S_p(G)=1$, we have that $K$ is non-$p$-solvable.
    Now, take $H$ to be a $p$-nilpotent maximal subgroup of $G$, which we know that exists.
    Since $H$ cannot contain $N$, then $G=NH$.
    On the other hand, we have $K\cap N=1$, and then
 $$K\cong KN/N\leq G/N\cong H/(N\cap H).$$
But this implies that $K$ is $p$-nilpotent, a contradiction. Therefore, the claim is proved.

\medskip
   Write $N=S_1\times \cdots \times S_n$, with $n\geq 1$ where $S_i$ are isomorphic  nonabelian simple groups.
     By Lemma \ref{maximal}, we can assert that $G$ possesses at least a non-$p$-nilpotent maximal subgroup $M$ not containing $N$.
 Thus, by hypothesis $M$ has prime index in $G$, say $|G:M|=q$, so $G/{\rm core}_G(M)$ is isomorphic with a subgroup of ${\rm Sym}(q)$.
  Now, the uniqueness of $N$ and the fact that $N\not \leq M$ imply that core$_G(M)=1$.
   As a consequence, $q$ is the largest prime dividing $|G|$ and $|G|_q=q$. Since $|G:M|=|N:N\cap M|$, we deduce that $q$ divides $|N|$,
    so  $n=1$, that is, $N$ is simple and has order divisible by $p$.
  Furthermore,  $G$ is almost simple, that is, $N\leq G\leq {\rm Aut}(N)$.
   Accordingly, we can apply Theorem \ref{almostsimple} to get the thesis of the theorem.
   \end{proof}

\begin{example}
An easy example of a non-simple group meeting the hypotheses of Theorem C(i) is the central extension of ${\rm PSL}_2(11)$, that is,  ${\rm SL}_2(11)\cong 2.{\rm PSL}_2(11)$.
Examples of  groups in case (i) whose $5$-solvable radical is non-central are easily constructed.
It suffices to consider the direct product $G\times H$, where $G\cong {\rm SL}_2(11)$ or ${\rm PSL}_2(11)$,
and $H$ is any nonabelian supersolvable group of order relatively prime to $|G|$, that is, to $\{2,3,5,11\}$.
On the other hand, not all groups $G$ satisfying $G/S_5(G)\cong {\rm PSL}_2(11)$ have to meet the hypotheses of this theorem.
The direct product ${\rm PSL}_2(7)\times {\rm PSL}_2(11)$ shows it.

 Most importantly, the list of simple groups given in Theorem C(ii) is not exhaustive.
Depending on the value of $p$, there exist distinct (projective special linear) simple groups whose maximal subgroups are either $p$-nilpotent or have prime index.
 For instance for $p=3$, the group ${\rm PSL_3}(2)$ satisfies (ii).
  For $p=5$, we already know that of ${\rm PSL}_2(11)$ satisfies our conditions.
   Moreover, the maximal subgroups of ${\rm PSL}_5(2)$ either have index $31$, or are isomorphic to $C_{31}\rtimes C_5$, which is $5$-nilpotent,
    or are $5'$-groups, and hence, are $5$-nilpotent too (see \cite{Con}).
     Thus, ${\rm PSL}_5(2)$ also satisfies our assumptions for $p=5$.

We can  illustrate another type of examples for primes $p\neq 3,5$. Take $p=19$ and the group ${\rm PSL}_3(q)$ with $q= 2^9$.
According to  \cite[Table 8.3]{BHR},  the  maximal subgroups of ${\rm PSL}_3(q)$ are isomorphic to either  $\hat{ }$ $[q^2]:{\rm GL}(2,q)$,
 which has index prime (precisely $q^2+q +1 =262657$), or are isomorphic to $\hat{ }$ $(q-1)^2: {\rm Sym}(3)$,  $\hat{}$ $(q^2+q+1).3$ or ${\rm PSL}_3(2^3)$,
  all of which trivially  have normal $19$-complement because they are $19'$-subgroups.
  Similar examples can be constructed for the family of groups ${\rm PSL}_n(q)$ by taking appropriate values of $n$, $q$ and  the prime $p$, whenever $(q^n-1)/(q-1)= p$.
\end{example}

We finish this section by proving Theorem D.

\begin{proof}[Proof of Theorem D]
  As $G$ is $p$-solvable we will use the upper $p'p$-series of $G$ to compute $l_p(G)$.
  We can assume that $K={\bf O}_{p',p,p'}(G)<G$, otherwise $l_p(G)\leq 1$ and we are finished.
  Let $\overline{G}=G/K\neq 1$.
   Notice that every maximal subgroup $\overline{M}$ of $\overline{G}$ satisfies that $M$ is not $p$-nilpotent  because $M$ contains $K$, which is not $p$-nilpotent, so by hypothesis,
    $|G:M|$ is prime.
    We can assert then that all maximal subgroups of $\overline{G}$ have prime index,
    and this is well known to be equivalent to  $\overline{G}$ being supersolvable \cite[Theorem 9.4.4]{Robinson}.
     In particular, $\overline{G}/F(\overline{G})$ is abelian, where $F(\overline{G})$ denotes the Fitting subgroup of $\overline{G}$. However,
  $$F(\overline{G})\leq {\bf O}_p(\overline{G})\times {\bf O}_{p'}(\overline{G})= {\bf O}_{p}(\overline{G}),$$
   so we have $F(\overline{G})= {\bf O}_{p}(\overline{G})=
\overline{{\bf O}_{p',p,p',p}(G)}$.
 As a consequence, we obtain that $G/{\bf O}_{p',p,p',p}(G)$ is abelian.
 This forces  $G/{\bf O}_{p',p,p',p}(G)$ to be a $p'$-group (possibly trivial), or equivalently, ${\bf O}_{p',p,p',p,p'}(G)=G$. Thus, $l_p(G)\leq 2$ and the proof is finished.
\end{proof}

\begin{remark}
As in Remark 3.2, we can employ \cite[Satz VI.6.6(a)]{Hup} to get an immediate consequence of  Theorem D:
 the nilpotency class of the Sylow $p$-subgroups of a group satisfying the hypotheses is at most $2$.
\end{remark}

\begin{example}  The $p$-length of $G$ in Theorem D cannot be further reduced. Let $G=C_3^3\rtimes {\rm Alt}(4)$, with ${\rm Alt}(4)$ acting faithfully on $C_3^3$.
Indeed, $G={\sf SmallGroup}(324,160)$ in the {\sf SmallGroups} Library of GAP \cite{GAP}.
This group satisfies $l_3(G)=2$  and has exactly three types of maximal subgroups.
Two of them correspond, respectively, to subgroups isomorphic to ${\rm Alt}(4)$  and the wreath product $C_3\wr C_3$, and both are $3$-nilpotent;
and the third type is a (normal) maximal subgroup isomorphic to $C_3^2\rtimes D_{12}$ ({\sf SmallGroup}(108,40)) where $D_{12}$ is acting via $D_{12}/C_3\cong C_2^2$.
This maximal subgroup is not $3$-nilpotent but has index $3$ in $G$. Accordingly, $G$ satisfies the conditions of Theorem D.
\end{example}

\begin{remark} The class of groups whose maximal subgroups are $p$-nilpotent or have prime index (for an arbitrary prime $p$) include, of course, $p$-nilpotent groups (and thus $p'$-groups),
 and also minimal non-$p$-nilpotent groups and supersolvable groups (in which every maximal subgroup has prime index).
 A minimal non-$p$-nilpotent group is a Schmidt group, and it easily turns out that ${\bf O}_{p,p'}(G)=1$, so have $p$-length $1$.
 On the other hand, the fact that $G/F(G)$ is abelian whenever $G$ is a supersolvable group leads to ${\bf O}_{p',p,p',p}(G)=G$ for every prime $p$   (including $p=2$).
 Thus, all supersolvable groups have $p$-length at most $2$.
We also stress, however,  that the converse statement of Theorem D is not true, i.e.,  not every $p$-solvable group of $p$-length less than  or equal to $2$ verifies the hypotheses of this theorem.
 The easiest example is ${\rm Sym}(4)$, with $l_3(G)=1$, which has non-$3$-nilpotent maximal subgroups, isomorphic to  ${\rm Sym}(3)$, of index $4$.
\end{remark}

\noindent
{\bf Acknowledgements}
This work is supported by the National Nature Science Fund of China (No. 12471017 and No. 12071181) and the  first named author is  also supported by Generalitat Valenciana, Proyecto CIAICO/2021/193. The  second  named author is also supported by the  Natural Science Research Start-up Foundation of Recruiting Talents of Nanjing University of Posts and Telecommunications (Grant Nos. NY222090, NY222091).

\medskip
\noindent
{\bf Data availability} Data sharing not applicable to this article as no data sets were generated or analyzed during
the current study.

\bigskip
\noindent
{\bf \large Declarations}

\medskip
\noindent
{\bf Conflict of interest} The authors have no conflicts of interest to declare.

\bibliographystyle{plain}

\begin{thebibliography}{1}




 \bibitem{BHR}
  Bray, J.,  Holt, D., Roney-Dougal, C.M.: The Maximal Subgroups of the Low-Dimensional Finite Classical Groups,
 407, London Math. Soc. Lecture Note Ser., 407, Cambridge University Press, Cambridge (2013)



\bibitem{Con}
Conway J.H., Curtis,  R.T., Norton, S.P., Parker, R.A. and Wilson, R.A.:
\newblock  Atlas of finite groups.
 \newblock  Oxford Univ. Press, London, (1985).

 \bibitem{DM}
Demina, E.N., Maslova, N.V.: \newblock  Nonabelian composition factors of a finite group with arithmetic constraints on nonsolvable maximal subgroups,
 \newblock{\em Tr. Inst. Mat. Mekh.},  20, No. 2, 122-134 (2014); translation in {\em Proc. Steklov Inst. Math.} 289 (2015) S64-S76.  https://doi.org/10.1134/S0081543815050065


\bibitem{Dickson}
Dickson, L.E.:
 Linear groups, with an Exposition of the Galois Field Theory,
Teubner, Leipzig, 1901.

\bibitem{GAP} The GAP Group, GAP - Groups, Algorithms and
Programming, Vers. 4.12.2; 2022. (http://www.gap-system.org)

\bibitem{Gura}
Guralnick, R.M.:
\newblock  Subgroups of prime power index in a simple group.
\newblock {\em J. Algebra}, 1983, 81(2): 304-311.




 \bibitem{Hup}
Huppert, B.:
 \newblock Endliche Gruppen I.
 \newblock Springer, Berlin, (1967).

\bibitem{HB}
Huppert, B., Blackburn, H.:
\newblock  Finite groups II, III.
Springer, Berlin, (1982)


\bibitem{King}
King, O.H.:
\newblock The subgroup structure of finite classical groups in terms of geometric configurations
London Math. Soc. Lecture Note Ser., 327
\newblock Cambridge University Press, Cambridge, 2005, 29-56.
 https://doi.org/10.1017/CBO9780511734885.003

\bibitem{LG}
Li, Q., Guo, X.:
\newblock  On $p$-nilpotence and solubility of groups.
\newblock {\em Arch. Math.} (Basel) {\bf 96} (1), 1-7, (2011). DOI: 10.1007/s00013-010-0215-0


\bibitem{Liebeck}
Liebeck, M.W., Praeger, C.E., Saxl, J.: A classification of the maximal subgroups of the finite alternating and symmetric groups. {\em J. Algebra} {\bf 111}, 365-383, (1987).

\bibitem{Liebeck2}
P. Kleidman, M. Liebeck: The Subgroup Structure of the Finite Classical Groups, vol. 129. London
Math. Soc. Lecture Note Ser., Cambridge University Press, Cambridge (1990).


 \bibitem{LPZ}
 Lu, J., Pang, L., Zhong, X.: Finite groups with non-nilpotent maximal subgroups. {\em Monatsh. Math.} {\bf 171} (2013) 425-431. DOI: 10.1007/s00605-012-0432-7



\bibitem{Robinson}
 Robinson, D.J.:
 A course in the theory of groups, 2nd ed., Springer-Verlag, New
York-Heidelberg-Berlin, (1996).

\bibitem{SB}
Shao, C.G., Beltr\'an, A.
Finite groups whose maximal subgroups are $2$-nilpotent or normal. {\em Bull. Malays. Math. Soc.} 47 (2024), No. 5, 8 pp.  DOI: 10.1007/s40840-024-01743-y

\bibitem{SLT}
Shi, J., Liu, W., Tian, Y.:
A note on finite groups in which every non-nilpotent maximal subgroup has prime index.
{\em J. Algebra Appl.} {\bf 23}, 9, (2024).
https://doi.org/10.1142/S0219498824501354

\bibitem{YJK}
 Yi, X., Jian, S., Kamornikov, F.:
Finite groups with given non-nilpotent maximal subgroups of prime index.
{\em J. Algebra Appl. } {\bf 18}, 5 (2019). https://doi.org/10.1142/S0219498819500877









\end{thebibliography}

\end{document}